\documentclass[12pt]{article}
\usepackage{amsfonts,amsmath,amssymb}

\newtheorem{theorem}{Theorem}[section]
\newtheorem{lemma}[theorem]{Lemma}
\newtheorem{corollary}[theorem]{Corollary}
\newtheorem{prop}[theorem]{Proposition}

\newenvironment{proof}{{\em Proof.}}{\hspace*{\fill}$\Box$\par\vspace{4mm}}

\usepackage{tikz}
\usetikzlibrary{shapes.geometric,positioning,calc}

\def\BigOh{\mathop{\rm O}}

\def\x{\times}

\begin{document}

\title{A lower bound for the smallest eigenvalue of a graph and an application to the associahedron graph}

\author{
Sebastian M. Cioab\u{a}\footnote{Department of Mathematical Sciences, University of Delaware, Newark, DE 19716-2553, USA, {\tt cioaba@udel.edu}. This research has been partially supported by NSF grant CIF-1815922 and a JSPS Invitational Fellowship for Research in Japan S19016.} \hspace{1cm} Vishal Gupta\footnote{Department of Mathematical Sciences, University of Delaware, vishal@udel.edu.}}

\date{\today}
\maketitle
\begin{abstract}
In this paper, we obtain a lower bound for the smallest eigenvalue of a regular graph containing many copies of a smaller fixed subgraph. This generalizes a result of Aharoni, Alon, and Berger in which the subgraph is a triangle. We apply our results to obtain a lower bound on the smallest eigenvalue of the associahedron graph, and we prove that this bound gives the correct order of magnitude of this eigenvalue. We also survey what is known regarding the second-largest eigenvalue of the associahedron graph.
\end{abstract}

\section{Introduction}

Our graph notation is standard, see \cite{BH} for undefined terms or notations. The {\em eigenvalues} of a graph $G=(V,E)$ are the eigenvalues of its adjacency matrix $A=A(G)$. For a graph $G$ with $n$ vertices and $\ell\geq 1$, denote by $\lambda_{\ell}(G)$ the $\ell$-th greatest eigenvalue of $G$ and let $\lambda^{\ell}(G)=\lambda_{n-\ell+1}(G)$ be its $\ell$-th smallest eigenvalue. Let $\lambda_{min}(G)$ denote the smallest eigenvalue $\lambda^1(G)$. The smallest eigenvalue of a graph is related to its chromatic number and independence number \cite{BH,GM} and has close connections to the max-cut of the graph \cite{AS,BCIM,GW,K}. Since the spectrum of a connected graph is symmetric if and only if the graph is bipartite (see \cite[Section 3.4]{BH} for example), it is natural to think of $\lambda_{min}(G)$ as a measure of the bipartiteness of $G$ (see \cite{Tre}). Aharoni, Alon, and Berger \cite{AAB} obtained a lower bound for the smallest eigenvalue of a regular graph where each vertex is contained in many triangles (see also \cite{CEG}). Knox and Mohar \cite{KM} obtained a lower bound for the smallest eigenvalue using graph decompositions and their work leads to a simpler proof of a result of Qiao, Jing, and Koolen \cite{QJK} on the smallest eigenvalue of a distance-regular graph.

In Section \ref{sec:oddcycle}, we obtain the following lower bound for the smallest eigenvalue of a regular graph.
\begin{theorem}\label{decK}
Let $K=(V,E)$ be a $k$-regular graph with $v$ vertices. Let $G$ be a $d$-regular graph having a collection $\mathcal{K}$ of subgraphs isomorphic to $K$ such that each vertex of $G$ is contained in at least $m$ copies of $K$ and each edge of $G$ is contained in at most $t$ copies of $K$. Then
\begin{equation}
d+\lambda_{min}(G)\geq (k+\lambda_{min}(K))\cdot \frac{m}{t}.
\end{equation}
\end{theorem}
This result implies the bound of Aharoni, Alon, and Berger \cite{AAB}. We will use the following corollary to find a lower bound for the smallest eigenvalue of the associahedron graph. 
\begin{corollary}\label{oddcycle} Let $d\geq 3$ and $m,r,t\geq 1$ be integers. Let $G$ be a $d$-regular graph having a collection $\mathcal{C}$ of subgraphs isomorphic to $C_{2r+1}$ such that each vertex of $G$ is contained in at least $m$ cycles of length $2r+1$ and each edge of $G$ is contained in at most $t$ cycles of length $2r+1$. Then
\begin{equation}
d+\lambda_{min}(G)\geq 4\sin^2\left(\frac{\pi}{4r+2}\right)\cdot \frac{m}{t}.
\end{equation}
\end{corollary}

In Section \ref{sec:assoc1}, we discuss the flip graph on the triangulations of a convex $n$-gon, also known as the associahedron graph $\mathcal{A}_n$. Let $n\geq 4$ and consider a convex $n$-gon $P$ whose vertices are labeled $1,2,\dots, n$. The set of vertices $T_n$ of $\mathcal{A}_n$ consists of the triangulations of $P$ with $n-3$ non-crossing diagonals. Two distinct triangulations are adjacent if they share $n-4$ diagonals. Equivalently, each neighbor of a triangulation $T$ can be obtained by flipping one of its diagonals (deleting one of its  diagonals, creating a quadrilateral in which one adds the other diagonal). The associahedron graph $\mathcal{A}_n$ is $1$-skeleton of the associahedron, an $(n-3)$-dimensional convex polytope that arises in many areas of mathematics \cite{FZ,MaPi,Po} and is also known as the Stasheff polytope \cite{Stasheff} or the Tamari lattice \cite{Tamari}. The graph $\mathcal{A}_n$ is $(n-3)$-regular and its number of vertices equals the Catalan number $C_{n-2}=\frac{\binom{2n-4}{n-2}}{n-1}$. The combinatorial properties of $\mathcal{A}_n$ have been investigated by several authors. Lucas \cite{Lucas} showed that $\mathcal{A}_n$ is Hamiltonian when $n\geq 5$. Lee \cite{Lee} proved that the automorphism group of $\mathcal{A}_n$ is the dihedral group of order $2n$. Pournin \cite{Pou} determined its diameter and showed that it equals $2n-10$ for $n>12$, confirming a conjecture of Sleator, Tarjan, and Thurston \cite{SlTaTh}.  Molloy, Reed, and Steiger \cite{MRS} studied the properties of the usual Markov chain/random walk on $\mathcal{A}_n$ in which one starts at a vertex and then selects a neighbor uniformly at random. Some of their results were improved by McShine and Tetali \cite{McST} and more recently by Eppstein and Frishberg \cite{Eppstein}.

In \cite{FFHHUW}, Fabila-Monroy, Flores-Penaloza, Huemer, Hurtado, Urrutia, and Wood study the chromatic number of various flip graphs such as the flip graph on perfect matchings of the complete graph $K_{2n}$ (see \cite{CRT} for related results) and the associahedron graph $\mathcal{A}_n$ for $n\geq 5$. The chromatic number of the associahedron graph $\mathcal{A}_n$ is obtained by computer in \cite{FFHHUW} and equals $3$ for $5\leq n\leq 9$ and $4$ when $n=10$. We have confirmed these computations. In \cite{FFHHUW}, the authors conjecture that the chromatic number $\chi(\mathcal{A}_n)\rightarrow \infty$ as $n\rightarrow \infty$ and that $\chi(\mathcal{A}_n)=O(\log n)$. The second conjecture was proved recently by Addario-Berry, Reed, Scott, and Wood \cite{ABRSW}, but the first conjecture is still open. Since $\chi(\mathcal{A}_n)\geq 1+\frac{n-3}{|\lambda_{min}(\mathcal{A}_n)|}$ (see  \cite[Theorem 3.6.2]{BH} or  \cite{WHaemers}), proving that $|\lambda_{min}(\mathcal{A}_n)|=o(n)$ would imply the conjecture from \cite{FFHHUW}. 

In this paper, we show that this is not the case and actually $|\lambda_{min}(\mathcal{A}_n)|=\Theta(n)$. The graph $\mathcal{A}_n$ is an induced subgraph of the Johnson graph $J\left(\frac{n(n-3)}{2},n-3\right)$. The eigenvalues of the Johnson graph are known (see \cite{BCIM,Delsarte}). Using Loday \cite{Loday}, one can also observe that the graph $\mathcal{A}_n$ is an induced subgraph of the simplicial rook graph $SR\left(n-2,\binom{n-1}{2}\right)$ introduced by Martin and Wagner \cite{MW} (see also \cite{BCHV}). We have not been able to use these facts to calculate the eigenvalues of $\mathcal{A}_n$. Instead, we will use Corollary \ref{oddcycle} and Cauchy eigenvalue interlacing to prove the following results. For $n\geq 5$, we show that 
\begin{equation}\label{lowerbndassoc}
\lambda_{min}(\mathcal{A}_n)\geq \frac{-5-\sqrt{5}}{8}(n-3)-\frac{3-\sqrt{5}}{8}.
\end{equation}
Using eigenvalue interlacing and computations of the smallest eigenvalue of $\mathcal{A}_n$ for $n\leq 12$, we prove that
\begin{equation}\label{upperbndassoc}
\lambda_{min}(\mathcal{A}_n)\leq -0.6904n+c_r,
\end{equation}
where $c_r$ is some constant that depends on the value of the remainder $r$ of $n$ when divided by $10$. We also show that the limit $\lim_{n\rightarrow \infty} \frac{\lambda_{min}(\mathcal{A}_n)}{n-3}$ exists and 
\begin{equation}\label{limit}
-0.6904\geq \lim_{n\rightarrow \infty} \frac{\lambda_{min}(\mathcal{A}_n)}{n-3}\geq \frac{-5-\sqrt{5}}{8}\approx -0.9045.
\end{equation}

\section{Proof of Theorem \ref{decK}}\label{sec:oddcycle}

We will use the following lemma.
\begin{prop}\label{K}
Let $K=(V,E)$ be a $k$-regular graph. For any vector $x\in \mathbb{R}^V$,
\begin{equation*}
\sum_{ij\in E}(x_i+x_j)^2\geq (k+\lambda_{min}(K))\sum_{\ell\in V}x_{\ell}^2.
\end{equation*}
\end{prop}
\begin{proof}
Let $A$ be the adjacency matrix of $K$.
Because $x^TAx\geq \lambda_{min}x^Tx$, it follows that 
$$
\sum_{ij\in E}(x_i+x_j)^2=x^T(kI+A)x\geq (k+\lambda_{min}(K))\sum_{\ell\in V}x_{\ell}^2.
$$\end{proof}
We now give the proof of Theorem \ref{decK}.

\begin{proof} Let $x$ be an eigenvector of euclidean norm one corresponding to $\lambda_{min}(G)$. If $A$ denotes the adjacency matrix of $G$, then 
\begin{equation*}
d+\lambda_{min}(G) =x^T(dI+A)x=\sum_{uv\in E(G)}(x_u+x_v)^2. 
\end{equation*}
For each edge $uv$, let $c_{uv}$ denote the number of copies of $K$ from $\mathcal{K}$ that contain $uv$. From our hypothesis, $c_{uv}\leq t$. For a vertex $w$, let $c_{w}$ denote the number of copies of $K$ from $\mathcal{K}$ containing $w$. Then $c_{w}\geq m$. For $H\in \mathcal{K}$, denote
$$\sigma(H)=\sum_{uv\in E(H)}(x_u+x_v)^2.$$ 
Proposition \ref{K} implies that
\begin{equation*}
\sigma(H)\geq (k+\lambda_{min}(K))\sum_{w\in V(H)}x_w^2.
\end{equation*}
Summing up over all the graphs in $\mathcal{K}$, we get that
\begin{align*}
\sum_{H\in \mathcal{K}}\sigma(H)&\geq (k+\lambda_{min}(K))\sum_{H \in \mathcal{K}}\sum_{w\in V(H)}x_w^2=
(k+\lambda_{min}(K))\sum_{w\in V(G)}c_wx_w^2\\
&\geq (k+\lambda_{min}(K))m.
\end{align*}
On the other hand,
\begin{align*}
\sum_{H\in \mathcal{K}}\sigma(H)&=\sum_{H\in \mathcal{K}}\sum_{uv\in E(H)}(x_u+x_v)^2=\sum_{uv\in E(G)}c_{uv}(x_u+x_v)^2\\
&\leq t\sum_{uv\in E(G)}(x_u+x_v)^2=t(d+\lambda_{min}(G)).
\end{align*}
Combining these last two inequalities gives the desired result.
\end{proof}
Corollary \ref{oddcycle} follows by taking $K=C_{2r+1}$. Taking $K=K_3, t=d-1$, one gets Theorem 1.1 from \cite{AAB} restricted to regular graphs.

\section{Proof of inequality \eqref{lowerbndassoc}}\label{sec:assoc1}

The graph $\mathcal{A}_n$ does not contain any triangles, but it contains cycles of length $5$ and we take advantage of this fact and use Corollary \ref{oddcycle} to obtain a lower bound for $\lambda_{min}(\mathcal{A}_n)$. First we need to show that each vertex of $\mathcal{A}_n$ is contained in at least $n-4$ cycles of length $5$ and each edge is contained in at most $4$ cycles of length $5$. 

Let $T$ be a vertex of $\mathcal{A}_n$. It corresponds to a triangulation of the $n$-gon into $n-2$ triangles using $n-3$ non-crossing diagonals. A triangle from this triangulation is called an ear if two of its sides are the sides of the $n$-gon and is called interior if all its sides are diagonals. We can associate a tree to $T$ as follows: the vertices correspond to the triangles of $T$ and two triangles are adjacent if and only if  they share one side/diagonal (see \cite[Theorem 1.5.1]{St}). We observe that this tree has $n-2$ vertices and each vertex of it has degree $1, 2$, or $3$. Vertices of degree one correspond to the ears of the triangulation and vertices of degree three correspond to the interior triangles. 
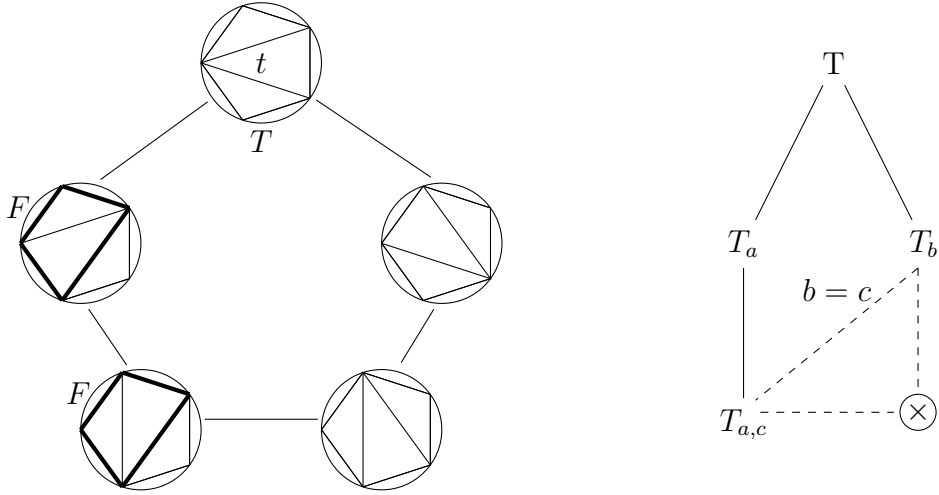
\begin{figure}
  \begin{minipage}{.5\textwidth}
    \centering
\begin{tikzpicture}[scale = 0.8]
\draw (0,0) circle[radius=1];
\draw (36:1 cm) -- (180:1 cm)--(324:1);
 \foreach \x in {-36,36,...,252} {
        \draw[fill] (\x:1 cm) -- (\x + 72:1 cm);}   
        
\begin{scope}[xshift = -3 cm, yshift = -3cm]
\draw (0,0) circle[radius=1];
\draw (36:1 cm) -- (180:1 cm);
\draw [ultra thick] (36:1cm) -- (252:1cm);
 \foreach \x in {-36,36,...,252} {
        \draw[fill] (\x:1 cm) -- (\x + 72:1 cm);}
\foreach \aa in {36, 108, 180}{
 \draw[ultra thick] (\aa:1cm)--(\aa+72:1cm);
}
\draw node at (150: 1.2){$F$};
        \end{scope}
        
\begin{scope}[xshift = 3 cm, yshift = -3cm]
\draw (0,0) circle[radius=1];
\draw (180:1 cm)--(324:1)--(108:1cm);
 \foreach \x in {-36,36,...,252} {
        \draw[fill] (\x:1 cm) -- (\x + 72:1 cm);}
        \end{scope}
        
\begin{scope}[xshift = -2 cm, yshift = -6.1cm]
\draw (0,0) circle[radius=1];
\draw[ultra thick] (36:1cm) -- (252:1cm);
\draw (252:1cm)--(108:1cm);
 \foreach \x in {-36,36,...,252} {
        \draw[fill] (\x:1 cm) -- (\x + 72:1 cm);}
        \foreach \aa in {36, 108, 180}{
 \draw[ultra thick] (\aa:1cm)--(\aa+72:1cm);
}
\draw node at (150:1.2){$F$};
        \end{scope}
        
\begin{scope}[xshift = 2 cm, yshift = -6.1cm]
\draw (0,0) circle[radius=1];
\draw (252:1cm)--(108:1cm)--(-36:1cm);
 \foreach \x in {-36,36,...,252} {
        \draw[fill] (\x:1 cm) -- (\x + 72:1 cm);}
        \end{scope}

\draw (216:1.1)--(216:3.3);
\draw (-34:1.1)--(-34:3.4);
\draw (235:5) -- (246:5.5);
\draw(-55:5) -- (-65:5.5);
\draw(-81:6) -- (261:6);
\draw node at (-90:1.3) {$T$};
\draw node at (0,0){$t$};
\end{tikzpicture}
\end{minipage}
  \begin{minipage}{.5\textwidth}
    \centering
 \begin{tikzpicture}[scale =1.6]
    \node {T}
    child{node {$T_a$}
    child{node{$T_{a,c}$}}}
    child{node {$T_b$}};
\draw[dashed] (0.7,-1.7)--node[above=0.3]{$b=c$} (-0.65,-2.8);
\draw [dashed](0.7,-1.7)--(0.7, -2.7);
\draw[dashed] (-0.61, -2.9) -- (0.5, -2.9);
\draw (0.7,-2.9) circle[radius = 0.15];
\draw node at (0.7,-2.9){$\times$};
\end{tikzpicture}
 \end{minipage}
 \caption{Left: a 5-cycle in $\mathcal{A}_n$ containing $T$. We depict the $n$-gon as a circle for simplicity. Right: there can not be a 5-cycle in $\mathcal{A}_n$ which is not of the form as shown in figure on the left.}
\label{fig1} 
\end{figure}
\begin{lemma}\label{pentagon}
 A cycle of length $5$ in $\mathcal{A}_n$ that contains a triangulation $T$, corresponds to two incident diagonals of $T$ whose removal creates a pentagon in $T$.   
\end{lemma}
\begin{proof}
  To see this, we take a triangle, say $t$, in $T$ of degree 2 or 3. Then by flipping the incident diagonals, as shown in Figure \ref{fig1} (left), we get a 5-cycle in $\mathcal{A}_n$. Now suppose there is 5-cycle in $\mathcal{A}_n$ containing $T$ that is not of this form. Let $T_a$, $T_b$ be the neighbors of $T$ in the 5-cycle, where $T_i$ denotes the triangulation we get after flipping the diagonal $i$ of $T$ and let $T_{a,c}\sim T_a$ in the 5-cycle, see Figure \ref{fig1} (right). If $b =c$, then $T_{a,b}\sim T_b$, but since $\mathcal{A}_n$ does not contain a triangle, we get a contradiction. If $b\neq c$, then there can not exist a triangulation which is adjacent to both $T_b$ and $T_{a,c}.$    
\end{proof}
\begin{prop}\label{vertexc5}
Let $n\geq 5$ and $T$ be a vertex/triangulation of $\mathcal{A}_n$. If $t_1$ equals the number of ears of $T$, then $T$ is contained in $n-6+t_1\geq n-4$ cycles of length $5$ in $\mathcal{A}_n$.
\end{prop}
\begin{proof}
Denote by $G=(V,E)$ the tree associated with $T$. For $j\in \{1,2,3\}$, let $t_j$ denote the number of vertices of degree $j$ in $G$. Since there are $n-2$ triangles in $T$, $t_1+t_2+t_3=n-2$. The Handshaking Lemma implies that $t_1+2t_2+3t_3=2(n-3)$. Therefore, $t_3=t_1-2$ and $t_2=n-2t_1$.


By Lemma \ref{pentagon}, the number of cycles of length $5$ containing $T$ equals $\sum_{v\in V}\binom{d_v}{2}$, where $d_v$ denotes the degree of the vertex $v$ in $G$. It is not hard to see that the previous expression is the same as $t_2+3t_3=n-2t_1+3(t_1-2)=n-6+t_1$. Since $G$ is a tree, it has at least two leaves and therefore $t_1\geq 2$. This finishes our proof.
\end{proof}

\begin{prop}\label{edgec5}
Let $n\geq 5$. If $T$ and $T'$ are two adjacent vertices in $\mathcal{A}_n$, then the edge $TT'$ is contained in at least one and at most four cycles of length $5$ in $\mathcal{A}_n$.
\end{prop}
\begin{proof}
Because $T$ and $T'$ are adjacent in $\mathcal{A}_n$, they have $n-4$ diagonals in common. Consider the two diagonals from the symmetric difference of $T$ and $T'$. They are the diagonals of a $4$-gon $F$. A cycle of length $5$ containing the edge $TT'$ in $\mathcal{A}_n$ corresponds to a side of $F$ (see Figure \ref{fig1}) that is a diagonal of $T\cap T'$ (or equivalently, not a side of the $n$-gon $P$). The $4$-gon $F$ can have at least one and at most four such sides. This finishes our proof.
\end{proof}

Combining Theorem \ref{oddcycle} with Proposition \ref{vertexc5} and Proposition \ref{edgec5}, we obtain inequality \eqref{lowerbndassoc}.

\section{Proof of inequality \eqref{upperbndassoc}}\label{sec:upperbnd}

We start with a simple observation.
\begin{prop}
If $k,\ell\geq 4$, then 
\begin{equation}
\lambda_{min}(\mathcal{A}_{k+\ell})\leq \lambda_{min}(\mathcal{A}_k)+\lambda_{min}(\mathcal{A}_{\ell}).
\end{equation}
\end{prop}
\begin{proof}
Recall that the Cartesian or box product $H\Box K$ of two graphs $H=(V,E)$ and $K=(W,F)$ has vertex set $V\times W$ and $(a_1,b_1)$ is adjacent to $(a_2,b_2)$ if $a_1\sim a_2$ and $b_1=b_2$ or $a_1=a_2$ and $b_1\sim b_2$. The adjacency matrix of $H\Box K$ equals $A(H)\otimes I_W+I_V\otimes A(K)$ and therefore the eigenvalues of $A(H\Box K)$ are of the form $\theta+\tau$, where $\theta$ is an eigenvalue of $A(H)$ and $\tau$ is an eigenvalue of $A(K)$. In particular, the smallest eigenvalue of $H\Box K$ equals $\lambda_{min}(H)+\lambda_{min}(K)$.

Recall that the vertices of the $k+\ell$-gon $P$ are labeled $1,2,\dots, k+\ell$ in clockwise direction. Consider the subgraph induced by the triangulations containing the diagonal connecting vertex $1$ to vertex $k$. This subgraph is isomorphic to $\mathcal{A}_k\Box \mathcal{A}_{\ell+2}$. Using Cauchy eigenvalue interlacing (see \cite{BH} for example) and the previous paragraph, we deduce that $\lambda_{min}(\mathcal{A}_{k+\ell})\leq \lambda_{min}(\mathcal{A}_k)+\lambda_{min}(\mathcal{A}_{\ell+2})$. It is not too hard to see that $\lambda_{min}(\mathcal{A}_n)$ is decreasing with $n$ (use Cauchy interlacing and the fact that $\mathcal{A}_{n}$ is an induced subgraph of $\mathcal{A}_{n+1}$ for $n\geq 4$), we get the desired result.
\end{proof}

The Fekete/subadditivity lemma (see \cite{Fekete} or \cite[Lemma 11.6]{vLW}) now implies that the following limit exists: 
$$\lim_{n\rightarrow \infty}\frac{\lambda_{min}(\mathcal{A}_n)}{n}.$$ 

For $n\leq 12$, we computed below the smallest eigenvalue of $\mathcal{A}_n$ rounded up to the first three decimal points.

\begin{center}
\begin{tabular}{|c|c|c|c|c|c|c|c|c|}
\hline
$n-3$ & 2 & 3 & 4&5 & 6 & 7 & 8 & 9 \\
\hline
$\lambda_{min}$ & -1.618  & -2.414  & -3.177 &-3.912 & -4.667 &-5.409 &-6.157&-6.904\\
\hline
\end{tabular}
\end{center}

Let $n=10(k+1)+2$ for $k\geq 1$. Recall that the vertices of the $n$-gon $P$ are labeled $1,2,\dots, n$ in clockwise direction. Consider the subgraph induced by the triangulations containing the diagonal connecting vertex $1$ to vertex $12$. It is not too hard to see that this subgraph is isomorphic to the box product $\mathcal{A}_{12} \Box\mathcal{A}_{n-10}$. Using Cauchy interlacing and the previous paragraph, we have that
\begin{align*}
\lambda_{min}(\mathcal{A}_n)&\leq \lambda_{min}(\mathcal{A}_{12} \Box\mathcal{A}_{n-10})=\lambda_{min}(\mathcal{A}_{12})+\lambda_{min}(\mathcal{A}_{n-10})\\
& \leq -6.904+\lambda_{min}(\mathcal{A}_{n-10}).
\end{align*}
Repeating this argument for $n-10,n-20,\dots,22$, we get that 
\begin{equation*}
\lambda_{min}(\mathcal{A}_n)\leq -6.904\times \frac{n-2}{10}=-0.6904(n-2)=-0.6904n+1.3808.
\end{equation*}
Similar upper bounds can be obtained when $n=10(k+1)+r$ for other values of $r$ between $0$ and $9$. The results in this section and in the previous section imply the inequalities in \eqref{limit}.

\section{The second eigenvalue of $\mathcal{A}_n$}

Molloy, Reed, and Steiger \cite{MRS} studied the properties of the random walk on $\mathcal{A}_n$ in which one starts at a vertex and then selects a neighbor uniformly at random. These authors proved that for any subset $S\subset T_n$ with $|S|\leq |T_n|/2$, there is a matching between $S$ and its complement $\overline{S}$ having at least $\frac{|S||\overline{S}|}{|T_n|n^{11}}$ edges leading to a lower bound of $\frac{1}{2n^{12}}$ for the conductance of $\mathcal{A}_n$. Molloy, Reed, and Steiger proved that at least $\Omega(n^{3/2})$ and at most $\BigOh(n^{23}\log n)$ steps are sufficient to get close (within $\epsilon$ in variation distance) to the stationary distribution (which is the uniform distribution over the vertices of $\mathcal{A}_n$). McShine and Tetali \cite{McST} improved the upper bound to $\BigOh(n^5\log(n/\epsilon))$ and recently, Eppstein and Frishberg \cite{Eppstein} further improved the upper bound to $\BigOh(n^{4.75})$.

Denote $\lambda_2=\lambda_2(\mathcal{A}_n)$. For $\epsilon\in (0,1)$, let $\tau(\epsilon)$ denote the mixing time of the Markov chain on $\mathcal{A}_n$ (see \cite{McST}, \cite[p.61]{Sinclair}), then
\begin{equation}\label{eq7}
\frac{n-3}{n-3-\lambda_2}\log(C_{n-2}/\epsilon)\geq \tau(\epsilon)\geq \frac{\lambda_2}{2(n-3-\lambda_2)}.
\end{equation}
 
For $n\leq 12$, we computed the second eigenvalue of $\mathcal{A}_n$ rounded down to the first three decimal points.
\begin{center}
\begin{tabular}{|c|c|c|c|c|c|c|c|c|}
\hline
$n-3$ & 2 & 3 & 4&5 & 6 & 7 & 8 & 9 \\
\hline
$\lambda_2$ & 0.618 & 2 & 3.231 & 4.383 & 5.488 & 6.564 & 7.622 & 8.667\\
\hline
\end{tabular}
\end{center}

It seems that the second eigenvalue of $\mathcal{A}_n$ tends to $n-3$. Aldous \cite{Aldous1} proved the following result and the proof below is a reformulation due to Vishesh Jain.
\begin{theorem} There is a positive constant $c$ such that $\lambda_2(\mathcal{A}_n)\geq (n-3)-\frac{c}{\sqrt{n}}.$
\end{theorem}
\begin{proof}
Note that the assertion is equivalent to the statement that the spectral gap $\gamma = 1-\lambda_{2}$ of the aforementioned random walk on $\mathcal{A}_{n}$ is $O(n^{-3/2})$. Let $\pi$ denote the uniform distribution on $\mathcal{A}_n$. By standard Markov chain theory (see, e.g., Lemma 13.7 in \cite{LP}), it suffices to exhibit a non-constant function $f : \mathcal{A}_{n} \to \mathbb{R}$ for which
\begin{align*}
    \frac{\mathbb{E}_{X_0, X_1}[(f(X_1) - f(X_0))^2]}{\text{Var}_{\pi}(f)} = O(n^{-3/2}),
\end{align*}
where $(X_0, X_1)$ are consecutive steps of the random walk on $\mathcal{A}_{n}$ with the initial state $X_0$ distributed according to $\pi$. 

This follows by a slight modification of the $\Omega(n^{3/2})$ lower bound on the mixing time of the random walk, due to Molloy, Reed, and Steiger \cite{MRS}; we refer the reader to Section 3 in their paper \cite{MRS} for the terminology used in the remainder of the proof. For $\tau \in \mathcal{A}_{n}$, let $f(\tau)$ denote the minimum distance between any vertex of the central triangle of $\tau$ and the point $p_{\lfloor n/4 \rfloor}$. By symmetry considerations, $\text{Var}_{\pi}(f) = \Theta(n^2)$, so it remains to show that
\[\mathbb{E}_{X_0, X_1}[(f(X_1) - f(X_0))^2] = O(\sqrt{n}).\]
For this, let $\mathcal{E}$ denote the event that one of the edges involved in the central triangle of $X_0$ is chosen to be flipped, and note that $f(X_1) - f(X_0) = 0$ on the complement of $\mathcal{E}$. Since $\mathbb{P}_{X_0,X_1}[\mathcal{E}] = \Theta(1/n)$, we have that
\begin{align*}
    \mathbb{E}_{X_0, X_1}[(f(X_1) - f(X_0))^2] 
    &= \mathbb{E}_{X_0, X_1}[(f(X_1) - f(X_0))^2 \mid \mathcal{E}]\cdot \mathbb{P}_{X_0, X_1}[\mathcal{E}]\\
    &= O(\mathbb{E}_{X_0, X_1}[(f(X_1) - f(X_0))^2 \mid \mathcal{E}] \cdot n^{-1})\\
    &= O(n^{3/2}\cdot n^{-1}),
\end{align*}
where the last line follows using the same computation as the one below Equation (4) in \cite{MRS}.
\end{proof}


Aldous \cite{Aldous1} conjectures  that the relaxation time of the random walk is $\BigOh(n^{3/2})$. This is equivalent to that there exists a positive constant $c'$ such that $\lambda_2(\mathcal{A}_n)\leq (n-3) - \frac{c'}{\sqrt{n}}.$ By (\ref{eq7}), we also get that the above random  walk mixes in  $O(n^{3/2}\log|\mathcal{A}_n|)$ which is $\BigOh(n^{2.5})$.

\section{Final remarks}\label{sec:final}

Our arguments in Section \ref{sec:oddcycle} and Section \ref{sec:assoc1} use the cycle $C_5$ since the associahedron graph $\mathcal{A}_5$ is isomorphic to $C_5$. A natural questions is to see what happens when $\mathcal{A}_5$ is replaced by $\mathcal{A}_6$. The graph $\mathcal{A}_6$ is $3$-regular and has the following eigenvalues (the exponents below are the multiplicities):
$$3^{(1)}, 2^{(2)}, \sqrt{3}^{(1)}, 0^{(2)}, (1-\sqrt{2}) ^{(3)}, -1^{(1)}, -\sqrt{3}^{(1)}, (-1-\sqrt{2})^{(3)}.$$
By a similar argument to Proposition \ref{vertexc5}, one can prove the following results.
\begin{prop}\label{vertexA6}
Let $n\geq 6$ and $T$ be a vertex of $\mathcal{A}_n$. 
The number of subgraphs of $\mathcal{A}_n$ that are isomorphic to $\mathcal{A}_6$ and contain $T$ equals the number of connected subgraphs with four vertices in the dual tree of $T$. 
\end{prop}
\begin{proof}
A subgraph isomorphic to $\mathcal{A}_6$ that contains $T$ is the same as a collection of three diagonals of $T$ whose deletion creates a hexagon. These three diagonals correspond to a connected subgraph with four vertices (or three edges) in the dual tree of $T$.
\end{proof}

If $H=(W,F)$ is the dual tree of the triangulation $T$, then a connected subgraph of $W$ with four vertices is either a path $P_4$ or a star $K_{1,3}$. The number of $P_4$s equals $\sum_{xy\in F}(d_x-1)(d_y-1)$.  It is fairly straightforward to show that this sum is minimized with $H$ is the path $P_{n-2}$ for which it equals $n-5$. Hence, every triangulation $T$ is contained in at least $n-5$ subgraphs isomorphic to $\mathcal{A}_6$.

\begin{prop}\label{edgeA6}
Let $n\geq 6$. If $T$ and $T'$ are two adjacent vertices in $\mathcal{A}_n$, then the edge $TT'$ is contained in at least one and at most fourteen subgraphs of $\mathcal{A}_n$ that are isomorphic to $\mathcal{A}_6$.
\end{prop}
\begin{proof}
The edge $TT'$ corresponds to $n-4$ non-intersecting diagonals in the polygon $P$. These diagonals partition the interior of the polygon into one quadrilateral $Q$ and $n-4$ triangles. There is at most one triangle neighboring the quadrilateral on each of its four sides. Therefore, there are at most $\binom{4}{2}=6$ ways to choose two of these triangles to obtain a hexagon containing $Q$. Each of these four triangles could have two triangles neighboring them. Thus, $Q$ could also be contained in $4\times 2$ other hexagons. 
\end{proof}

Using the results of this section, one can obtain that 
$$
(n-3)+\lambda_{min}(\mathcal{A}_n)\geq \frac{(2-\sqrt{2})(n-5)}{14}.
$$

Unfortunately, this seems to be a worse estimate than our lower bound in \eqref{lowerbndassoc}. We hope that our methods for bounding the smallest eigenvalue of $\mathcal{A}_n$ can be used for other families of graphs. We finish our paper with a natural open problem, namely 
determining the limit 
$$\lim_{n\rightarrow\infty}\frac{\lambda_{min}(\mathcal{A}_n)}{n-3}.$$

\section*{Acknowledgments} We thank Orest Bucicovschi, Florian Frick, Chris Godsil, Jack Koolen, Sabrina Lato, Zhao Kuang Tan, and Prasad Tetali for their comments and suggestions. We are grateful for Vishesh Jain for his explanation of the results in Section 5.

\end{document}